\documentclass{amsart}
\usepackage{amsfonts}
\usepackage{amsmath} 
\usepackage{xcolor}

\newtheorem{theorem}{Theorem}
 
\newtheorem{corollary}[theorem]{Corollary}
\newtheorem{definition}[theorem]{Definition}

\newtheorem{lemma}[theorem]{Lemma} 

\newtheorem{proposition}[theorem]{Proposition}
\newtheorem{remark}[theorem]{Remark}

\newcommand{\eps}{\varepsilon}
 \newcommand{\RR}{{\mathbb R}}

\title[Uniform controllability of a fourth-order equation]{Uniform null controllability of a fourth-order parabolic equation with a transport term}   
\author[M. L\'{o}pez-Garc\'{i}a and A. Mercado]{} 
\thanks{
The first author was partially supported by project PAPIIT IN100919 of DGAPA, UNAM.
The second author  was partially supported  by FONDECYT (Chile) grant  1171712 and  BASAL Project, CMM - U. de Chile.}
 
\email{marcos.lopez@im.unam.mx}
\email{alberto.mercado@usm.cl }
\begin{document}  

\maketitle 
\centerline{\scshape  Marcos L\'opez-Garc\'{\i}a }
\medskip
{\footnotesize

 \centerline{Instituto de Matem\'{a}ticas-Unidad Cuernavaca}
   \centerline{Universidad Nacional Aut\'{o}noma de M\'{e}xico}
   \centerline{ Apdo. Postal 273-3, Cuernavaca Mor. CP 62251, M\'exico}
}
\medskip

\centerline{\scshape Alberto Mercado}
\medskip
{\footnotesize
 \centerline{  Departamento de Matem\'{a}tica}
 \centerline{ Universidad T\'{e}cnica Federico Santa Mar\'{i}a,  
 Casilla 110-V, Valpara\'{i}so, Chile; }
 \centerline{ Institut de Math\'ematiques de Toulouse, UMR 5219,
 France}
}

%
%

\begin{abstract}   
In this paper we prove a uniform controllability result for a fourth order parabolic partial differential equation which includes a transport term, 
when  the coefficients of higher order terms vanish.
We prove the null controllability of the system with a single boundary control, 
and also it is obtained that the cost of the control decreases to zero, 
under the hypothesis of the control time being large enough.
Moreover, we prove that, if the control time is small enough, the cost  of controllability increases to infinity. 
\end{abstract}

\section{Introduction}
The controllability  of parabolic second order partial differential equations   has deserved  a lot of attention in the literature.
The boundary null controllability  for the heat equation was  proved by Fattorini and Russell in \cite{FR} using the moment method. 
Afterward, the distributed null controllability of the heat equation in higher dimensions was proved  by  Lebeau and Robbiano \cite{LR}  and Fursikov and Imanuvilov  \cite{FI}. Since these seminal results, a large number of articles have been devoted to those  subjects. 

Fourth-order parabolic equations  have also been studied recently. 
In  \cite{C} the  boundary controllability for the one-dimensional case is proved  using the moment method and  \cite{CM} 
is devoted to the boundary controllability 
of a nonlinear fourth order parabolic equation using  global Carleman estimates.
We mention also  \cite{GK},  \cite{LRR} and \cite{Tak} for results in higher dimensions.

The problem we deal with  in this paper consists in the estimation of the cost of null controllability of a 
fourth order parabolic partial differential equation when the   coefficients of the  higher order terms vanish.
An analogous 
problem for second order equations  has been considered in several works:
For  the heat equation with vanishing viscosity coefficient, we can cite 
 \cite{CoronG}, \cite{LebGuerr}, where are used Carleman estimates and  energy  estimates. These results were improved in 
  \cite{ErvDard}, \cite{Glass},  \cite{Lissy2015} and  \cite{Lissy2017},  using different techniques, mainly the moment method.  
  In particular, in  \cite{Glass} some complex-analytic methods are used in order to obtain uniform controllability. 

Concerning the cost of controllability for fourth order equations, as far as we know, the only reference is \cite{CaG}, where it is studied 
a parabolic equation posed in $(0,L) \times (0,T)$, and composed by a transport term with
constant  velocity $M$ and a  fourth order term with vanishing viscosity. 
It is proved that, under the hypothesis
$ T \geq  40 L/|M|$
and using two boundary controls,  
the cost of controllability remains uniformly bounded with respect to the diffusion coefficient. 
The main tools used in that article are Carleman estimates and diffusion estimates. 

The objective of this work consists in adapting  the   approach introduced in \cite{Glass}, in order to obtain uniform controllability of a  fourth order parabolic system, using a single boundary control.

More precisely, given $\varepsilon > 0$,  we  consider the  system 
\begin{equation}
\label{EqCon}
\left\{\begin{array}{rl}
y_{t}+ \varepsilon y_{xxxx}  +  \delta y_{xxx} + My_x =0, & (t,x)\in (0,T)\times(0,L), \\
y(t,0)= 0, \hspace{0.1cm} y(t,L)=0, & t\in(0,T), \\
By(t,0)= u(t), \hspace{0.1cm} By (t,L) =0, & t\in(0,T), \\
y(0,x)=y_{0}(x), & x\in(0,L),
\end{array}\right.
\end{equation}
where   
\begin{equation} \label{del}
\delta = \delta(\varepsilon,M):=- 2 \varepsilon^{2/3} M^{1/3}
\end{equation} 
and 
\begin{equation} \label{Bda}
By :=   \varepsilon y_{xx} + \frac \delta2 y_{x}
=  \varepsilon y_{xx} -  \varepsilon^{2/3} M^{1/3}.
\end{equation}

The precise value of $\delta$ in the system above allow us to obtain an explicit sequence of eigenfunctions of the adjoint system (see equation \eqref{funprop} below).

In the next result we state the null-controllability of the system for each $\varepsilon>0$.

\begin{theorem}\label{Teo0}
Given any $M \in \RR^*=\RR \backslash \{0\}$, 
and for each $L, T, \varepsilon > 0$, system \eqref{EqCon} is null-controllable in $L^2(0,L)$. 
This is, for each $y_0\in L^2(0,L)$ there exists a control  $u\in L^2(0,T)$ such that the corresponding solution $y$ of \eqref{EqCon} 
 satisfies $y(T,\cdot)=0$. 
 \end{theorem}
\bigskip  

For $y_0\in L^2(0,L)$ consider the set $U( \varepsilon, T,L,M,y_0)$ of all controls $u\in L^2(0,T)$ such that the corresponding solutions $y$ of system (\ref{EqCon}) satisfy $y(T,\cdot)=0$. 
From Theorem \ref{Teo0} we have that this set is nonempty. Then we define the cost of the null controllability of system \eqref{EqCon} as 
\begin{equation} \label{costo}
K=K( \varepsilon, T,L,M):=\sup_{\|y_0\|_{L^2(0,L)}\leq 1}\left\{\inf{\|u\|_{L^2(0,T)}:u\in U( \varepsilon, T,L,M,y_0)} \right \}. 
\end{equation}

Notice that scaling arguments yield the relations
$$K\left(\eps,aT,a^{1/4}L,\frac{M}{a^{3/4}}\right)=a^{3/8}K(\eps,T,L,M)\quad \text{ and}$$  
$$a^{1/8}K\left(a\eps,T,a^{1/4}L,a^{1/4}M\right)=K(\eps,T,L,M)$$
for every $\varepsilon, L,T, a>0$ and $M\in \mathbb{R}^*.$\\
%

We recall that a transport equation is controllable if and only if $T > L/ |M|$. 
Due to the presence of the transport term in equation \eqref{EqCon},
we directly obtain,  
if $T$ does not satisfy the above inequality,   an  estimate 
for the asymptotic  cost of the control when $\varepsilon$ vanishes, as the following result shows.
See Remark \ref{powere} below regarding the 
factor
 $\frac{1}{\sqrt \varepsilon}$, 
 and Remark \ref{RemTr} for the connection with the transport equation. 
\begin{proposition}  \label{MainT3}
		Given $L,T>0$  such that  
				\begin{equation}\label{Tpeq}
		T  <  L/|M|,
		\end{equation}
		we have  that
		\begin{equation} \label{inft1}
\frac{1}{\sqrt \varepsilon} K( \varepsilon, T,L,M) \rightarrow \infty.
\end{equation}
as $\varepsilon \rightarrow 0^+$.
\end{proposition} 
%
%

On the other hand, for large $T$, we can expect that the cost of the control remains bounded when
$\varepsilon \to 0$. This is the main question we intend to answer in this work, and the following is the
 main results in that direction.  

\begin{theorem} \label{MainT}
	Let  $L,T>0$ be such that 
	\begin{equation}\label{restriction}
		T >4.57 L/M, \, M>0;\quad T >6.19 L/|M|, \, M<0.
		\end{equation}
Then there exist $c,C>0$ such that 
\begin{equation} \label{KMax}
K( \varepsilon, T,L,M) \leq C\exp (-c \varepsilon^{-1/3})
\end{equation}
for each $\varepsilon \in (0,1)$.
	\end{theorem}

Theorem \ref{MainT} implies 
the uniform null controllability of system (\ref{EqCon}) provided  that (\ref{restriction}) holds. In particular, we get  that
$$K( \varepsilon, T,L,M)\rightarrow 0$$
 as $\varepsilon \rightarrow 0^+$.

Moreover, we are able to give a precise lower bound for the cost $K$ for some particular small times. 
\begin{theorem} \label{MainT2}
Let $L,T>0$ be such that 
\begin{equation}\label{restlow}
		T  < 0.33  L/M, \, M>0;\quad T < 1.69 L/|M|, \, M<0.
		\end{equation}
				Then 
		 there exist $c,C>0$ such that 
\begin{equation} \label{Klow}
K( \varepsilon, T,L,M) \geq C \exp (c \eps^{-1/3})
\end{equation}
for each $\varepsilon \in (0,1)$.

	\end{theorem}

The rest of the article is organized as follows. In Section \ref{WP}, we prove the well-posedness of the proposed system by establishing a transposition scheme. Also, in that section we state the characterization of the controllability results, and we prove 
Proposition
\ref{MainT3}. In Section \ref{multi} we prove that the differential operator of  the adjoint equation is diagonalizable, and we introduce an entire function with simple zeros at those eigenvalues.
In Section \ref{OBS} it is proved the observability inequality 
which corresponds 
 to the uniform controllability result stated in Theorem \ref{MainT}, the main result of this work. 
 Finally, in Section \ref{lowbounds}  we prove Theorem \ref{MainT2}, by showing 
  the existence of lower bounds for the null control.

\section{Well-Posedness and characterizations} \label{WP}

\subsection{Well posedness}
In order to define the solution by transposition of the control system \eqref{EqCon}, we consider the adjoint equation given by
\begin{equation}
\label{EqAdj}
\left\{\begin{array}{rl} 
- \varphi_{t}+ \varepsilon \varphi_{xxxx}  -  \delta \varphi_{xxx} - M\varphi_x  =  g, & (t,x)\in (0,T)\times(0,L), \\
\varphi(t,0)= 0,\hspace{0.1cm} \varphi(t,L)=0, & t\in(0,T), \\
B^*\varphi(t,0)=  0, \hspace{0.1cm} B^*\varphi(t,L)=0, & t\in(0,T), \\
\varphi(T,x)= \varphi_0(x) , & x\in(0,L),
\end{array}\right.
\end{equation}
%
where $B^*\varphi =   \varepsilon \varphi_{xx} - (\delta/2) \varphi_x$.
After a change of variable in $t$, we get the equivalent system 
\begin{equation}
\label{EqAdj2}
\left\{\begin{array}{rl} 
 \varphi_{t}+ \varepsilon \varphi_{xxxx}  -  \delta \varphi_{xxx} - M\varphi_x  =  g, & (t,x)\in (0,T)\times(0,L), \\
\varphi(t,0)= 0,\hspace{0.1cm} \varphi(t,L)=0, & t\in(0,T), \\
B^*\varphi(t,0)=  0, \hspace{0.1cm} B^*\varphi(t,L)=0, & t\in(0,T), \\
\varphi(0,x)= \varphi_0(x) , & x\in(0,L).
\end{array}\right.
\end{equation}

We have the following well-posedness framework for system \eqref{EqAdj2} (and then of system \eqref{EqAdj} as well).
%
\begin{proposition} \label{WPad}
For each $g \in L^2(0,T; L^2(0,L))$ and $\varphi_0 \in L^2(0,L)$, the system \eqref{EqAdj2}
has   a unique solution 
$$\varphi \in C([0,T]; L^2(0,L)) \cap L^2(0,T; H^2\cap H_0^1(0,L)).$$
Moreover, there exists a constant $C>0$ independent of $\varepsilon$ such that 
\begin{equation} \label{desWP} 
\begin{aligned}
\| \varphi \|_{L^ \infty(0,T; L^2(0,L)) }  &  \leq C ( \|g\|_{L^2(L^2)} + \| \varphi_0 \|_{L^2(0,L)}), \\
\sqrt \varepsilon \,  \| \varphi \|_{L^2(0,T; H^2(0,L))}  & \leq C (\|g\|_{L^2(L^2)} + \| \varphi_0 \|_{L^2(0,L)})
\end{aligned}
\end{equation}
for each $g \in L^2(0,T; L^2(0,L))$ and $\varphi_0 \in L^2(0,L)$.
\end{proposition}
\begin{proof}
Assuming enough regularity, we multiply the first  equation of system  \eqref{EqAdj2} by $\varphi$ and we integrate in space. 
Taking into account the boundary conditions satisfied by $\varphi$, we obtain 
\begin{equation} \label{ener}
\frac 12 \frac{d}{dt} \int_0^L |\varphi(t, x)|^2 dx + \varepsilon \int_0^L  |\varphi_{xx}(t,x)|^2 dx =  \int_0^L g(t,x) \varphi(t,x)dx
\end{equation}
for each $t \in (0,T)$.  Hence, Gronwall inequality implies that
\begin{equation} \label{ener2}
\int_0^L |\varphi(t, x)|^2 dx 
\leq 
C\left(  \int_0^L  |\varphi_0(x)|^2 dx + \int_0^T \int_0^L | g(t,x) |^2 dxdt \right)
\end{equation}
for each $t \in (0,T)$.  Integrating \eqref{ener} in $(0,T)$ we deduce that 
\begin{equation} \label{ener3}
\varepsilon \int_0^T\int_0^L |\varphi_{xx}(t, x)|^2 dx dt
\leq 
C\left(  \int_0^L  |\varphi_0(x)|^2 dx + \int_0^T \int_0^L | g(t,x) |^2 dxdt \right).
\end{equation}

From \eqref{ener2} and \eqref{ener3} we obtain inequalities \eqref{desWP} for regular solutions. Using an standard density argument, we get the desired result.  
 \end{proof}

Using a trace regularity result, we obtain the following estimate for the trace corresponding to the observation of the control system.

\begin{corollary} \label{cotrace}
There exists a constant $C>0$ independent of $\varepsilon$ such that 
$$\sqrt \varepsilon \, \| \varphi_x(\cdot,0)  \|_{L^2(0,T)}   \leq C (\|g\|_{L^2(L^2)} + \| \varphi_0 \|_{L^2(0,L)})$$
for each $g \in L^2(0,T; L^2(0,L))$ and $\varphi_0 \in L^2(0,L)$, where 
$\varphi$ is the solution of  system \eqref{EqAdj2}.
\end{corollary}

Next, we state the definition of solutions of the control system by means of a transposition scheme. 

\begin{definition} \label{deftran} 
Given $y_0 \in L^2(0,L)$ and $u \in L^2(0,T)$, we say that 
$y = y(t,x)$ is a solution of
system \eqref{EqCon} if, for every $t \in [0,T]$, the function $y(t,\cdot) \in L^2(0,L)$ satisfies
\begin{equation} \label{transp0}
\int_0^L y(t,x) \varphi_0(x) dx 
= 
\int_0^Ly_0(x) \varphi(0,x) dx 
+ 
 \int_0^t u(s) \varphi_{x}(s,0) ds
\end{equation}
for every  $\varphi_0  \in L^2(0,L)$, where $\varphi$ is  the solution of system  \eqref{EqAdj} posed in $(0,t) \times (0,L)$ with 
$\varphi(t,\cdot) = \varphi_0$ and  $g=0$.

\end{definition}

\begin{proposition} \label{existun}
For each $y_0 \in L^2(0,L)$ and $u \in L^2(0,T)$, there exists a unique solution  
$y \in C([0,T]; L^2(0,L))$ 
of
system \eqref{EqCon}.  
Moreover, there exists  $C >0$ independent of $\varepsilon$ such that
\begin{equation} \label{maxy}
\max_{t \in [0,T]} \| y(t, \cdot) \|_{L^2(0,L)} \leq C \left(  \|y_0\|_{L^2(0,L)} + \frac{1}{\sqrt  \varepsilon}  \|u \|_{L^2(0,T)} \right).
\end{equation}
\end{proposition}

\begin{proof}
Let us take  $y_0 \in L^2(0,L)$ and $u \in L^2(0,T)$.
For each $t \in [0,T]$,  we consider  system \eqref{EqAdj}  in $(0,L) \times (0,t)$, with    $\varphi(t, \cdot) = \varphi_0 \in L^2(0,L)$ and
 $g = 0$.
 By Proposition \ref{WPad}, there exists a unique  solution $\varphi  \in C([0,t]; L^2(0,L)) \cap L^2(0,t; H^2 \cap H_0^1(0,L))$.
Hence, the  functional defined by
$$
\varphi_0  \in L^2(0,L) \mapsto   \int_0^L y_0(x) \varphi(0, x) dx  +  
\int_0^t u(s) \varphi_{x}(s, 0)ds 
$$ 
is linear and continuous.
%
The Riesz representation theorem implies that  there exists a unique  $y(t, \cdot)  \in L^2(0,L)$  satisfying \eqref{transp0}.

Moreover, from    \eqref{desWP}  and Corollary \ref{cotrace}, we  have that 
\begin{equation}
\begin{aligned}
\left|\int_0^L y(t,x) \varphi_0(x) dx \right|
& \leq  
\| y_0 \|_{L^2(0,L)} \| \varphi(0,\cdot) \|_{L^2(0,L)}  
+ 
\|u\|_{L^2(0,t)} \|  \varphi_{x}(\cdot,0) \|_{L^2(0,t)}  \\
& \leq C \left(  \|y_0\|_{L^2(0,L)} + \frac{1}{\sqrt \varepsilon} \|u \|_{L^2(0,T)} \right) \|\varphi_0\|_{L^2(0,L)}
\end{aligned}
\end{equation}
for each $\varphi_0 \in L^2(0,L)$, and hence we obtain the desired result.
\end{proof}
\begin{remark}
In particular, if $y$ is the solution of system \eqref{EqCon} with control $u$ and initial condition $y_0$, 
from Duhamel principle we get that 
\begin{equation} \label{transp}
\begin{split}
\int_0^T\int_0^L y(t,x) g(t,x) dx dt
+ \int_0^L  y(T,x) \varphi_0(x) dx
= \\
\int_0^Ly_0(x) \varphi(0,x) dx 
+ 
\int_0^T u(t) \varphi_{x}(t,0) dt 
\end{split}
\end{equation}
for all $\varphi_0 \in L^2(0,L)$, $g \in L^2(0,T;L^2(0,L))$, where $\varphi$ is  the solution of system  \eqref{EqAdj}.
\end{remark}
\begin{remark} \label{powere}
The parameter  $\varepsilon $ does not appear explicitly
in \eqref{transp0} 
 because we have 
set $By(\cdot,0)  = u $ in system \eqref{EqCon}. 
However, 
we recall that 
  the trace operator $By$ depends on  $\varepsilon$, and 
other  choices for the  control are posible.
For instance we can  define the control $\tilde u \in L^2(0,T)$  
acting in  system \eqref{EqCon} by the boundary condition 
$$\frac{1}{\sqrt \varepsilon} By(\cdot, 0) = \tilde u. $$

In that case, the corresponding solution  would  satisfies,  instead of \eqref{maxy}, that 
\begin{equation} \label{maxy2}
\max_{t \in [0,T]} \| y(t, \cdot) \|_{L^2(0,L)} \leq C \left(  \|y_0\|_{L^2(0,L)} +   \|\tilde u \|_{L^2(0,T)} \right),
\end{equation}
and then the coefficient $\frac {1}{\sqrt \varepsilon}$ on   Proposition \ref{MainT3}
 would not appear.
However, we have chosen the boundary condition in this way 
for the sake of simplicity of the duality condition. 
Also, we recall that the exponential cost  appearing in  \eqref{KMax} implies 
that taking this choice (or another  power of $\varepsilon$)   would not affect the result stated in Theorem \ref{MainT}.

\end{remark}

%
\subsection{Controllability}
Given the previous framework of well-posedness, we establish the notion of null controllability. 
\begin{definition}  \label{defcont}

Given $y_0 \in L^2(0,L)$, we say that  $u \in L^2(0,T)$ is a null-control of $y_0$ in system \eqref{EqCon} if 
\begin{equation} \label{defcon}
 \int_0^L   y_0(x) \varphi(0, x ) dx 
=  - 
 \int_0^T u(t) \varphi_{x}(t,0) dt 
\end{equation}
for each $\varphi_0 \in L^2(0,L)$, where $\varphi$ is  the solution of system  \eqref{EqAdj} with  $g  = 0$.
\end{definition}

With the previous definition in mind, the following characterization of null controllability is a direct consequence of a classical result of functional analysis. 

%
\begin{proposition} \label{ContObs}

System \eqref{EqCon} is null controllable in $L^2(0,L)$ if and only if there exists $C>0$ such that
\begin{equation}  \label{inobs}
\int_0^L |\varphi(0,x)|^2 dx \leq C  \int_0^T |\varphi_{x}(t,0)|^2 dt
\end{equation}
for each $\varphi_0 \in L^2(0,L)$, where $\varphi$ is the solution of system \eqref{EqAdj} with  $g  = 0$.
Furthermore, we have the relation
\begin{equation} \label{Kdual}
K( \varepsilon, T,L,M)  =   \inf \{ \sqrt C \, : \, \eqref{inobs} \text{ is satisfied for each } \varphi_0 \in L^2(0,L)  \}. 
\end{equation}

\end{proposition} 

In Section \ref{OBS} we will prove inequality \eqref{inobs} with precise estimations of
the constant $C$, obtaining as a consequence  corresponding   bounds for $K( \varepsilon, T,L,M)$.

\begin{remark}
A  natural related problem is  to consider the control in the zero-order boundary condition. This is, the  system 
\begin{equation}
\label{control0}
\left\{\begin{array}{rl}
y_{t}+ \varepsilon y_{xxxx}  +  \delta y_{xxx} + My_x =0, & (t,x)\in (0,T)\times(0,L), \\
y(t,0)= u(t), \hspace{0.1cm} y(t,L)=0, & t\in(0,T), \\
By(t,0)= 0, \hspace{0.1cm} By (t,L) =0, & t\in(0,T), \\
y(0,x)=y_{0}(x), & x\in(0,L).
\end{array}\right.
\end{equation}
It can be directly proved that 
System \eqref{control0} is null-controllable if and only if there exists $C>0$ such that
\begin{equation}  \label{inobs2}
\| \varphi(0,x)\|_X  \leq C  \|   \delta \varphi_{xx}(t,0) - \varepsilon   \varphi_{xxx}(t,0)\|_Y
\end{equation}
for all solutions of system \eqref{EqAdj}, in suitable  spaces $X$, $Y$.
This problem can also be studied with the methods of this work.
%
\end{remark}

\subsection{Convergence of the solutions}
Following, we prove Proposition \ref{MainT3}:  for small times,  the cost  of the control  
(we recall definition \eqref{costo})
has  the behavior  given by \eqref{inft1} when $\varepsilon \to 0^+$.
The proof follows directly from the duality with the adjoint equation and the basic  properties of the transport equation.

\begin{proof}[Proof of Proposition \ref{MainT3}] 
We deal only with the case $M>0$,  the proof is analogous for  $M<0$.
We suppose \eqref{inft1} is not true, 
that is, 
there exists a sequence $\varepsilon_n \to 0^+$ such that for 
any   $y_0 \in L^2(0,L)$ with $\|y_0\| \leq 1$
there exists  a null-control   $u_n \in L^2(0,T)$  for equation \eqref{EqCon} with $\varepsilon = \varepsilon_n$ and initial condition $y_0$, such that
 $$\left\{ \frac{1}{\sqrt \varepsilon_n} u_n\right\} \subset L^2(0,T)$$ 
 is uniformly bounded.  We denote by $y_n$ the corresponding  solution  of the system.

Then $y_n(T, \cdot) = 0$, and from \eqref{maxy}, we have that $\{y_n\} \subset C([0,T]; L^2(0,L))$ 
is bounded. Hence there exists a subsequence $\{y_{n_k}\}$ with 
\begin{equation} \label{convd1}
y_{n_k} \rightharpoonup  y \, \text{  in } \, L^2(0,T;L^2(0,L)).
\end{equation}

On the other hand, let $\varphi \in C^4((0,T) \times (0,L))$ be such that, for each $t \in [0,T]$,  $\text{supp} (\varphi(t , \cdot )) \subset (0,1)$ is compact. 
From \eqref{transp} we have that 
\begin{equation} \label{dualtr}
\begin{split}
\int_0^T\int_0^L y_{n_k}(t,x) (  - \varphi_{t}+ \varepsilon_{n_k} \varphi_{xxxx}  -  \delta_{n_k} \varphi_{xxx} - M\varphi_x ) dx dt
\\
= 
\int_0^Ly_0(x) \varphi(0,x) dx.
\end{split}
\end{equation}
Letting   $ k \to \infty$ in \eqref{dualtr}, taking into account \eqref{convd1}  and the fact that 
$\varepsilon_{n_k}, \delta_{n_k} \to 0$, we get
\begin{equation} \label{trantr}
\begin{split}
\int_0^T\int_0^L y(t,x) (  -  \varphi_{t} - M\varphi_x ) dx dt
= 
\int_0^Ly_0(x) \varphi(0,x) dx
\end{split}
\end{equation}
for any $y_0 \in L^2(0,L)$ such that $\|y_0\|_{L^2} \leq 1$.

From hypothesis \eqref{Tpeq} we can take $\varphi_0 \in C_0^\infty(\RR)$ with support contained in the interval $(TM,L)$ and such that $ 0 <\|\varphi_0\|_{L^2} \leq 1$. 
We define
$$\varphi(t,x) = \varphi_0( x+(T-t)M)$$ for each $t \in [0,T]$, $x\in [0,L]$. 
Then we have $-\varphi_t - M\varphi_x = 0$,
  $\text{supp} (\varphi(t , \cdot )) \subset (0,1)$ is compact  for each $t \in [0,T]$ and
  $\varphi(0,\cdot) \not\equiv 0$. However, taking $y_0 = \varphi(0,\cdot)$  in \eqref{trantr} we get $\varphi(0,\cdot)\equiv 0$. 
  From this contradiction we deduce that \eqref{inft1}  is true.
\end{proof}
\begin{remark} \label{RemTr}
From the proof of   Proposition \ref{MainT3} is deduced that the weak limit $y$ satisfies the transport equation 
\begin{equation} 
y_t + My_x = 0  
\end{equation}
in $(0,T) \times (0,L)$ with initial condition $y_0$, but no information for boundary conditions is retrieved when passing to  the limit 
 in system \eqref{EqCon}. 
 In particular, we have that, in the subdomain 
 $$\{ (x,t) \, : \,  tM < x <  L,  0< t < \min\{T, L/M\} \}$$ (for $M>0$),
 the solution of the equation   only depends on the initial condition, and it is 
 explicitly given  by $y(t,x) = y_0(x-tM)$.
\end{remark}

\section{Diagonalization and construction of the multiplier.} \label{multi}
The value of $\delta$ was chosen such that the spatial  differential operator $P$ defined  
by equation \eqref{EqAdj2}, 
\begin{equation} 
P:= \eps\partial_{xxxx}  + 2 \eps^{2/3}M^{1/3}\partial_{xxx}  -M\partial_x,
\end{equation}
is diagonalizable in $L^2(0,L) $.  In fact, its  eigenvectors are given by 
\begin{equation} \label{funprop}
e_k(x)=\exp\left(-\frac{M^{1/3}}{2\eps^{1/3}}x\right)\sin\left(\frac{k\pi x}{L}\right)
\end{equation}
with $k\in\mathbb{N}$,  being the  corresponding eigenvalues
\begin{equation}\label{lambdak}
\lambda_k:=\eps\left(\frac{k^2\pi^2}{L^2}+\frac{3M^{2/3}}{4\eps^{2/3}} \right)^2 -\frac{M^{4/3}}{4\eps^{1/3}}.
\end{equation}
We notice that $\left\{\sqrt{\frac 2L} e_k\right\} $ is a Hilbert basis of $L^2(0,L)$ for the scalar product
$$\langle u,v \rangle:= \int_0^L \exp\left(\frac{M^{1/3}}{\eps^{1/3}}x\right) u(x)v(x)dx. $$
Clearly, there exists a constant $C>0$ such that
\begin{equation}\label{normabase}
\|e_k\|_{L^2}\leq C,\,\, M>0 ;     \quad  \|e_k\|_{L^2}\leq C \exp\left(2^{-1}|M|^{1/3}L\varepsilon^{-1/3}\right),\,\, M<0.
\end{equation}
for all $k\geq 1$, where $\|\cdot\|_{L^2}$ denotes the usual norm in $L^2(0,L)$.\\

In order to study the controllability of our system, we will explicitly construct a biorthogonal family $\{\psi_k\}$ of the 
 exponentials $\{ \exp( -\lambda_k(T-t))\}$. 

With that objective in mind, we define a function $\Phi_{\varepsilon}$ having simple zeros exactly at $\{-i\lambda_k:k\in \mathbb{N}\}$ 
by
\begin{equation}\label{Phi}
\Phi_{\varepsilon}(z)=\frac{\sin\left(L\sqrt{\frac{1}{\eps^{1/2}}\sqrt{iz+\frac{M^{4/3}}{4\eps^{1/3}}}-\frac{3}{4}\frac{M^{2/3}}{\eps^{2/3}}}\right)}{L\sqrt{\frac{1}{\eps^{1/2}}\sqrt{iz+\frac{M^{4/3}}{4\eps^{1/3}}}-\frac{3}{4}\frac{M^{2/3}}{\eps^{2/3}}}}.
\end{equation}
From the inequality $|\sin z^{1/2}|\leq \exp(|z|^{1/2}/\sqrt{2})$ for all $z\in \mathbb{C}$, we obtain that 
\begin{equation}\label{phimodulo}
|\Phi_{\varepsilon}(z)|\leq  \frac{   \exp\left(\frac{L}{\sqrt{2}}\left(\frac{1}{\sqrt{2}}+\frac{\sqrt 3 }{2}\right)\frac{|M|^{1/3}}{\eps^{1/3}}\right)\exp\left( \frac{L}{\sqrt{2}\eps^{1/4}}|z|^{1/4}  \right)}{\left|L\sqrt{\frac{1}{\eps^{1/2}}\sqrt{iz+\frac{M^{4/3}}{4\eps^{1/3}}}-\frac{3}{4}\frac{M^{2/3}}{\eps^{2/3}}}\right|}, \quad z\in \mathbb{C}.
\end{equation}

We fix
\begin{equation}\label{auxiliares}
a:=\frac{T-\tau}{2\pi}, \quad \widetilde{L}:=\left(L  +\alpha \eps^{1/4}\right)(2+\sqrt{2})^{-1/2}, \quad \widehat{L}:=\frac{\sec(\pi/8)}{\sqrt{2}}\widetilde{L}+\alpha \eps^{1/4},
\end{equation}
with $\alpha, \tau $ positive numbers independent of $\eps$ to be chosen later.\\

We set
$$s(t)=at-\frac{\widetilde{L}}{\sqrt{2}\pi\cot(\pi/8)}\frac{t^{1/4}}{\eps^{1/4}}, \quad t>0.  $$

Using that
$$\int_0^\infty \log\left|1-\frac{x^2}{t^2} \right|dt^\gamma=|x|^{\gamma}\pi\cot\frac{\pi\gamma}{2} \text{ for } 0<\gamma <2,$$
we get that
\begin{equation}\label{logs}
\int_0^\infty \log\left|1-\frac{x^2}{t^2} \right|ds(t)=-\frac{\widetilde{L}}{\sqrt{2} \eps^{1/4}}|x|^{1/4}, \quad x\in \RR.
\end{equation}

The function $s(t)$ is increasing for
\begin{equation*}
t>A:=\left(\frac{\widetilde{L}}{2 \sqrt{2} (T-\tau)\cot(\pi/8)}\right)^{4/3}\eps^{-1/3},
\end{equation*}
and $s(B)=0$ for
\begin{equation*}
B:=\left(\frac{2\widetilde{L}}{\sqrt{2} (T-\tau)\cot(\pi/8)}\right)^{4/3}\eps^{-1/3}.
\end{equation*}

We set $d\nu(t)$ as the restriction of the measure $ds(t)$ to the interval $[B,\infty)$ and we introduce
the holomorphic function on $\mathbb{C}\backslash \mathbb{R}$ given by
\begin{equation}
g(z):=\int_0^{\infty}\log\left(1-\frac{z^2}{t^2}\right)d\nu(t)=\int_B^{\infty}\log\left(1-\frac{z^2}{t^2} \right)ds(t),
\end{equation}
and for $z\in\mathbb{C}$ we consider the harmonic function 
\begin{equation}
 U(z):=Re(g(z))=\int_0^{\infty}\log\left|1-\frac{z^2}{t^2}\right|d\nu(t)=\int_B^{\infty}\log\left|1-\frac{z^2}{t^2} \right|ds(t).
\end{equation}

As usual, $[\cdot]$ stands for the integer part function, so we set
\begin{equation*}
\nu(t):=\int_0^td\nu,
\end{equation*}
and for $z\in \mathbb{C}$ we define
\begin{equation}
\widetilde{U}(z):=\int_0^{\infty}\log\left|1-\frac{z^2}{t^2}\right|d[\nu(t)]=\int_B^{\infty}\log\left|1-\frac{z^2}{t^2} \right|d[s(t)],
\end{equation}
and we also consider
\begin{equation}
h(z):=\int_B^{\infty}\log\left(1-\frac{z^2}{t^2} \right)d[s(t)], \quad   z\in \mathbb{C}.
\end{equation}

Clearly,
$$U(z)=Re(g(z)), \,\, z\in \mathbb{C}\backslash \mathbb{R}\quad \text{ and }\quad\widetilde{U}(z)=Re(h(z)), \,\, z\in \mathbb{C}.$$

Let $(\mu_k)_{k\geq 1}$ be the sequence satisfying $s(\mu_k)=k$ for all $k\geq1$. 
Since $d[v]=\sum_k\delta_{\mu_k}$ and $\mu_k=O(k)$, we have that
\begin{equation}
\exp(h(z)) =\prod_{k\in \mathbb{N}}\left(1-\frac{z^2}{\mu_{k}^2}  \right),\quad z\in \mathbb{C},
\end{equation}
is an entire function. 

We will use the multiplier defined by 
\begin{equation}
f(z):=\exp(h(z -i)),  \quad z\in \mathbb{C}.
\end{equation}

From now on, $\Im z$ stands for the imaginary part of any $z\in\mathbb{C}$.
\begin{lemma} The function $U(z)$ is continuous on $\Im z\leq 0$. Moreover, $U(x)$ is an even continuous function on $\mathbb{R}$ such that
\begin{equation}\label{hipoisson}
\int_{-\infty}^{\infty}\frac{\log^+\left(\exp(U(x))\right)}{1+x^2}dx<\infty,
\end{equation}
and there exists a constant $C_1>0$ such that
\begin{equation}\label{Uinex}
U(x)\leq -\frac{\widetilde{L}}{\sqrt{2} \eps^{1/4}}|x|^{1/4}+C_1aB
\end{equation}
for all $x\in\mathbb{R} $. In fact,
\begin{equation}\label{C1}
C_1:=-\min_{x\in\mathbb{R}}\int_0^1\log\left|1-\frac{x^2}{t^2}\right|d(t-t^{1/4})\sim 6,55 <6,56.
\end{equation}
\end{lemma}
\begin{proof}
By using (\ref{logs}) we obtain that
$$U(x)=-\frac{\widetilde{L}}{\sqrt{2} \eps^{1/4}}|x|^{1/4}-\int_0^B\log\left|1-\frac{x^2}{t^2} \right|ds(t),$$
then we make the change of variable $t\mapsto t/B$ to see that
$$\int_0^B\log\left|1-\frac{x^2}{t^2} \right|ds(t)=aB\int_0^1\log\left|1-\frac{x^2}{B^2t^2} \right|d(t-t^{1/4}).$$
For $x>0$ we have
\begin{eqnarray*}
I(x):=\int_0^1\log\left|1-\frac{x^2}{t^2} \right|d(t-t^{1/4})&=&x\log\left|\frac{x+1}{x-1}\right|-(1+\sqrt{2})\pi\sqrt[4]{x}+\sqrt[4]{x} \log\left|\frac{\sqrt[4]{x}-1}{\sqrt[4]{x}+1}\right|\notag \\
&&+\frac{\sqrt[4]{x}}{\sqrt{2}}\log\left|   \frac{(\sqrt[4]{x}-1)^2+1}{(\sqrt[4]{x}+1)^2+1}\right|+\sqrt[4]{4x}\arctan(\sqrt[4]{x}-1)\notag \\ 
&&+\sqrt[4]{x}\arctan(\sqrt[4]{x}+1)+2\sqrt[4]{x}\arctan(\sqrt[4]{x}).
\end{eqnarray*}
For $x>0$ large enough we have

$$x\log\left|\frac{x+1}{x-1}\right|+\sqrt[4]{x} \log\left|\frac{\sqrt[4]{x}-1}{\sqrt[4]{x}+1}\right|+\frac{\sqrt[4]{x}}{\sqrt{2}}\log\left|   \frac{(\sqrt[4]{x}-1)^2+1}{(\sqrt[4]{x}+1)^2+1}\right|>0,$$
so there exist constants $c,R>0$ such that 
$$I(x)\geq -c \sqrt[4]{x}\quad \text{for all } x >R.$$
Since $U(x)=-\frac{\widetilde{L}}{\sqrt{2} \eps^{1/4}}|x|^{1/4}-aBI(\frac{x}{B})$ it follows that $U(x)<0$ for $x>R, 0<\varepsilon<\varepsilon_0$ with $\varepsilon_0>0$ small enough. This implies (\ref{hipoisson}).
\end{proof}

\begin{lemma}
For $\Im  z <0$, we have
\begin{equation}\label{rep}
U(z)=-\pi a\Im(z)-\frac{1}{\pi}\int_{-\infty}^{\infty}\frac{\Im(z)U(t)}{|z-t|^2}dt.
\end{equation}
\end{lemma}

\begin{proof}
We set $F(z)=\exp g(-z)$ for $z\in \mathbb{H}:=\{z \in \mathbb{C}: \Im z >0\}$. Clearly $F$ is a holomorphic function on $\mathbb{H}$, continuous on $\overline{\mathbb{H}}$ and having  no zeros in $\mathbb{H}$.
Moreover, the identities (\ref{A}), (\ref{B}) imply that
$$\log |F(z)|=U(-z)\leq \int_B^{\infty}\log(1+|z|^2/t^2)ds(t)\leq T |z|/2, \quad z\in \mathbb{H}.$$

Condition (\ref{hipoisson}) implies that we can consider the Poisson transform $V(z)$, $z\in \mathbb{H}$, of the function $\log^+\left(\exp(U(x))\right)$. \\
  
By integrating by parts and using the dominated convergence theorem we obtain
\begin{equation}\label{singular}
\limsup_{y\rightarrow \infty}\frac{\log|F(iy)|}{y}=\limsup_{y\rightarrow \infty}\frac{U(-iy)}{y}=2\lim_{y\rightarrow \infty}\int_{B/y}^{\infty}\frac{1}{1+\theta^2} \frac{s(y\theta)}{y\theta}d\theta=\pi a.
\end{equation}

The theorem in \cite[page 38]{koosis1} yields
$$v(z):=U(-z)-\pi a \Im(z)-V(z) \leq 0, \quad z\in \mathbb{H} .$$
The theorem in \cite[page 41]{koosis1} implies that there exist a constant $\alpha \geq 0$ and a positive measure $\mu (t)$, with $\int_{-\infty}^{\infty}d\mu(t)/(1+t^2)<\infty $, such that
$$-v(z)=\alpha \Im(z)+\frac{1}{\pi}\int_{-\infty}^{\infty}\frac{\Im(z)}{|z-t|^2} d\mu(t),\quad z\in \mathbb{H}.$$

The key point is that $U(z)$ is continuous on $\Im z\leq 0$, so the measure $d\mu(t)$ satisfies (see \cite[page 47]{koosis1})
$$\log^+(\exp \left(U(t)\right)-d\mu(t)=U(t)dt.$$

Thus,
$$U(-z)=(\pi a\Im(z)-\alpha)+\frac{1}{\pi}\int_{-\infty}^{\infty}\frac{\Im(z)U(t)}{|z-t|^2}dt, \quad z\in \mathbb{H}.$$
We use the last equality to compute $\limsup_{y\rightarrow \infty}U(iy)/y$, and (\ref{singular}) implies that $\alpha =0$.
\end{proof}

The following result can be found in \cite[page 162]{koosis2}.
\begin{lemma}
For any increasing function $\nu(t)$ with $\nu(t)=O(t)$ for $t>0$, we have
\begin{equation}\label{log+}
\int_0^{\infty}\log\left|1-\frac{z^2}{t^2} \right|(d[v(t)]-dv(t))\leq \log\left(\frac{\max(|x|,|y|)}{2|y|}+\frac{|y|}{2\max(|x|,|y|)}\right).
\end{equation}
where $z=x+iy$, $y \neq 0$.
\end{lemma}

\begin{lemma}
For $x\in \mathbb{R}$ we have
\begin{equation}\label{estimacionUtilde}
\widetilde{U}(x-i)\leq \log^{+}(|x|)+\pi a+aBC_1 -\frac{\widetilde{L}}{\sqrt{2} \eps^{1/4}}(2+\sqrt{2})^{1/2}x^{1/4}.
\end{equation}
\end{lemma}

\begin{proof}
From (\ref{Uinex}) and (\ref{rep}) we obtain
$$U(x-i)\leq \pi a+\frac{2aBC_1}{\pi}\arctan(\infty)-\frac{\widetilde{L}}{\sqrt{2} \pi \eps^{1/4}}\int_{-\infty}^{\infty}\frac{|t|^{1/4}}{1+|x-t|^2}dt.$$
Straightforward computations show that
\begin{eqnarray*}
\int_{-\infty}^{\infty}\frac{|t|^{1/4}}{1+|x-t|^2}dt&=&2\pi(1+x^2)^{1/8}\cos\left(\frac{1}{4} \arctan\left(\frac{1}{x} \right)\right)\\
&\geq & \pi (2+\sqrt{2})^{1/2}(1+x^2)^{1/8}.
\end{eqnarray*}

From (\ref{log+}) we have
\begin{eqnarray}\label{estimatildaU}
\widetilde{U}(x-i)&=& -\int_0^{\infty}\log\left|1-\frac{(x-i)^2}{t^2} \right|d(\nu(t)-[\nu(t)])+U(x-i)\\
&\leq & \log^{+}(|x|)+\pi a+aBC_1-\frac{\widetilde{L}}{\sqrt{2}\eps^{1/4}}(2+\sqrt{2})^{1/2}(1+x^2)^{1/8}.  \notag
\end{eqnarray}
\end{proof}

\begin{lemma}
We set
\begin{equation}
G(y):=\int_0^1\log\left|1+\frac{y^2}{t^2} \right|d(t-t^{1/4}), \quad y\in \mathbb{R}.
\end{equation}
Then,
\begin{equation}\label{belowUt}
\widetilde{U}(iy)\geq \pi a |y|-\frac{\widetilde{L}}{\sqrt{2} \eps^{1/4}}\sec(\pi/8)|y|^{1/4}-aBG\left(\frac{y}{B}\right)-\log\left(1+\frac{y^2}{B^2}\right), \quad y\in \mathbb{R}^*.
\end{equation}
\end{lemma}

\begin{proof}
We have the following identities
\begin{equation}\label{A}
\int_0^{\infty}\log\left|1+\frac{y^2}{t^2} \right|dt=\pi |y|,
\end{equation}
\begin{equation}\label{B}
\int_0^{\infty}\log\left|1+\frac{y^2}{t^2} \right|dt^{1/4}=\pi\csc\left(\frac{\pi}{8}\right)|y|^{1/4}.
\end{equation}
\begin{equation}\label{C}
\int_0^B\log\left|1+\frac{y^2}{t^2} \right|ds(t)=aBG\left(\frac{y}{B}\right).
\end{equation}

For instance, to get the second one we integrate by parts, thus
$$\int_0^{\infty}\log\left|1+\frac{y^2}{t^2} \right|dt^{1/4}=2y^2\int_0^{\infty}\frac{t^{-3/4}dt}{t^2+y^2}=\int_0^{\infty}\frac{8|y|^{1/4}dx}{1+x^8}=\pi\csc\left(\frac{\pi}{8}\right)|y|^{1/4}.$$
To obtain the equality (\ref{C}) we just make the change of variable $t\mapsto t/B$.\\

For any $y\in\mathbb{R}^*$ fixed, the function $\partial_t\left[\log\left(1+y^2/t^2 \right) \right] $ is negative, so we integrate by parts and we use that $s(B)=0$ to obtain
\begin{eqnarray*}
\int_B^{\infty}\log\left(1+\frac{y^2}{t^2}\right)d(s(t)-[s(t)])&=&\int_B^{\infty}-\partial_t\left[\log\left(1+\frac{y^2}{t^2}\right)\right](s(t)-[s(t)])dt\\ \notag
& \leq &\int_B^{\infty}-\partial_t\left[\log\left(1+\frac{y^2}{t^2}\right)\right]dt=\log\left(1+\frac{y^2}{B^2}\right),
\end{eqnarray*}
which implies
$$\widetilde{U}(iy) \geq  U(iy)-\log\left(1+\frac{y^2}{B^2}\right).$$
The result follows from the identities (\ref{A}), (\ref{B}) and (\ref{C}).
\end{proof}
\begin{remark}\label{eltipo}
We claim the function $\exp(h(z))$ is an entire function of exponential type $\pi a=(T-\tau)/2$. To see that we obtain a representation for $\widetilde{U} $ as in (\ref{rep}) for $z\in\mathbb{C}\backslash \mathbb{R}$, then we apply (\ref{estimatildaU}) together some straightforward computations. Hence $f(z)=\exp(h(z-i))$ is an entire function of exponential type $(T-\tau)/2$.

\end{remark}
\section{Observability inequality} \label{OBS} 
In this section we prove the main results of the paper.  
We will construct  a biorthogonal  family $\{\psi_k\}$ of the 
family of exponentials $\{\exp(-\lambda_k(T-t)) \}$. 
Then we obtain estimates for the norms  $\|\psi_k\|_{L^2(0,T)}$ and we deduce 
the desired  observability inequalities.

For $k\in \mathbb{N}$ we introduce the following entire function
$$\widetilde{J}_k^{\varepsilon}(z):=\frac{\Phi_{\varepsilon}(z)}{\Phi_{\varepsilon}'(-i\lambda_k)(z+i\lambda_k)}\frac{f(z)}{f(-i\lambda_k)}.$$
Clearly,
$$\widetilde{J}_k^{\varepsilon}(-i\lambda_j)=\delta_{k,j}, \quad j,k\geq 1.$$

Remark \ref{eltipo} and (\ref{phimodulo}) imply that $\widetilde{J}_k^{\varepsilon}$ is a entire function of exponential type $T/2$.\\

An easy computation shows that
\begin{equation*}
\Phi_{\varepsilon}'(-i\lambda_k)=\frac{(-1)^k}{4\eps \pi^2 k^2}\left( \frac{\pi^2 k^2}{L^2} +\frac{3}{4}\frac{M^{2/3}}{\eps ^{2/3}}  \right)^{-1}, \quad k\geq 1,
\end{equation*}
therefore
\begin{equation}
|\Phi_{\varepsilon}'(-i\lambda_k)|^{-1}\leq 4\eps^{1/3} \pi^2 k^2\left( \frac{\pi^2 k^2}{L^2} +\frac{3}{4}M^{2/3}  \right), \quad k\geq 1.  \label{phider}
\end{equation}

With the previous results, we are able to obtain estimates of the functions $\widetilde{J}_k^{\varepsilon}$. 
We will establish two estimates, a first one for all $T>0$, and a second one more precise and valid only  for $T$ large enough. 
\begin{proposition}
There exists small enough $\varepsilon_0 > 0$ such that 
\begin{enumerate}
\item  \label{incisa}
For each $T>0$ and  $0< \varepsilon < \varepsilon_0$, there exists a constant $C_{T, \varepsilon} >0$ such that 
\begin{equation}\label{burda}
| \widetilde{J}_k^{\varepsilon}(x)|  \leq  C_{T, \varepsilon} \frac{k^4 \exp \left( -\pi a\lambda_k +\frac{\widehat{L}}{\eps^{1/4}}|\lambda_k|^{1/4}\right)}{ |x^2+\lambda_k^2|^{1/2}}, \quad k\geq 1.
\end{equation}

\item If $T >  4L / |M|$, then  there exists a constant $C_T>0$ such that
$$ | \widetilde{J}_k^{\varepsilon}(x)|\leq C_T\frac{ \exp\left( \frac{L}{2}\left(1+\frac{\sqrt 3 }{\sqrt{2}}\right)\frac{|M|^{1/3}}{\eps^{1/3}}+aB(C_1-C_2)-\pi a\lambda_k +\frac{\widehat{L}}{\eps^{1/4}}|\lambda_k|^{1/4}\right)}{\varepsilon^{-2/3}k^{-4}|x^2+\lambda_k^2|^{1/2}}, \,\,k\geq 1.$$
where $C_1$ is given in \eqref{C1}, and
$$C_2:=-G\left((1+\sqrt{2})^22^{-5/3}5)\right)\sim 5,99.$$
\end{enumerate}
\end{proposition}
\begin{proof}
By (\ref{phimodulo}) we have that
$$|\Phi_{\varepsilon}(x)|\leq \beta(x)^{-1}  \exp\left(\frac{L}{2}\left(1+\frac{\sqrt 3 }{\sqrt{2}}\right)\frac{|M|^{1/3}}{\eps^{1/3}}\right)\exp\left( \frac{L}{\sqrt{2}\eps^{1/4}}|x|^{1/4}  \right),$$
where $$\beta(x)=\left|L\sqrt{\frac{1}{\eps^{1/2}}\sqrt{ix+\frac{M^{4/3}}{4\eps^{1/3}}}-\frac{3}{4}\frac{M^{2/3}}{\eps^{2/3}}}\right|.$$
Since
$$\beta(x)^2=\frac{L^2M^{2/3}}{2\varepsilon^{2/3}}\left|\sqrt{i\frac{4\varepsilon^{1/3}}{M^{4/3}}x+1}-\frac{3}{2} \right|\geq \frac{L^2M^{2/3}}{4\varepsilon^{2/3}}, \quad x\in \mathbb{R},$$
by using (\ref{estimacionUtilde}) we get
\begin{eqnarray} 
|\Phi_{\varepsilon}(x)\exp(\widetilde{U}(x-i))| & \leq & C_T \varepsilon^{1/3} \exp\left(\frac{L}{2}\left(1+\frac{\sqrt 3 }{\sqrt{2}}\right)\frac{|M|^{1/3}}{\eps^{1/3}}- \frac{(2+\sqrt{2})^{1/2}\widetilde{L}-L}{\sqrt{2}\eps^{1/4}}|x|^{1/4} \right. \notag
\\
& & \left. +aBC_1+\log^+(|x|) \right) \notag \\
& \leq & C_T\varepsilon^{1/3} \exp\left( \frac{L}{2}\left(1+\frac{\sqrt 3 }{\sqrt{2}}\right)\frac{|M|^{1/3}}{\eps^{1/3}}+aBC_1 \right). 
\label{estphi}
\end{eqnarray}

Since $\widetilde{U}(iy)$ is a decreasing function on $\mathbb{R}^-$ and using (\ref{belowUt}) we get
\begin{eqnarray*}
|f(-i\lambda_k)|&=& \exp\widetilde{U}(-i(1+\lambda_k))\geq  \exp\widetilde{U}(-i\lambda_k)    \notag\\ 
&=& \exp\left(\pi a\lambda_k -\frac{\widetilde{L}}{\sqrt{2} \eps^{1/4}}\sec(\pi/8)|\lambda_k|^{1/4}-aBG\left(\frac{\lambda_k}{B}\right)-\log\left(1+\frac{\lambda_k^2}{B^2}\right)\right).
\end{eqnarray*}
Assuming that $\alpha >\max\left(1,2^{-1/4}(1+\sqrt{2})^{3/2}(8!)^{3/8}T-L\right)$ we have that $B^2\geq 8!$, together with the inequality $\log(1+y^2/8!)\leq y^{1/4}$, $y>0$, we obtain
\begin{equation} 
|f(-i\lambda_k)|\geq \exp\left( \pi a\lambda_k -\frac{\widehat{L}}{\eps^{1/4}}|\lambda_k|^{1/4}-aBG\left(\frac{\lambda_k}{B}\right)\right).  \label{Belowf}
\end{equation}
Using (\ref{auxiliares}) we have
\begin{equation} \nonumber
\lambda_k/B \geq \left(\cot(\pi/8) |M|(T-\tau)/(\sqrt{2}\widetilde{L})\right)^{4/3}5/16 =: \delta_{T,\varepsilon}, \quad k\geq 1. 
\end{equation}
and then, as $G$ is a decreasing function we get 
\begin{equation} \label{delta}
|f(-i\lambda_k)|\geq \exp\left( \pi a\lambda_k -\frac{\widehat{L}}{\eps^{1/4}}|\lambda_k|^{1/4}-aBG(\delta_{T,\varepsilon}) \right)
\end{equation}
%
for all $k\in \mathbb{N}$ and  $\eps>0$ small enough. 
From  (\ref{phider}), \eqref{estphi}--\eqref{delta}  we get
$$| \widetilde{J}_k^{\varepsilon}(x)|\leq C_T\frac{ \exp\left( \frac{L}{2}\left(1+\frac{\sqrt 3 }{\sqrt{2}}\right)\frac{|M|^{1/3}}{\eps^{1/3}}
+aB(C_1  + G(\delta_{T,\varepsilon}))-\pi a\lambda_k +\frac{\widehat{L}}{\eps^{1/4}}|\lambda_k|^{1/4}\right)}{\varepsilon^{-2/3}k^{-4}|x^2+\lambda_k^2|^{1/2}},$$
from where we deduce item \eqref{incisa}.

On the other hand, if $|M|T/L>4$ then $|M|(T-\tau)/\widetilde{L}>4$ for small enough positive numbers $\varepsilon,\tau$; thus
$$
\delta_{T,\varepsilon}
\geq (1+\sqrt{2})^22^{-5/3}5.$$

From  (\ref{phider}), \eqref{estphi}--\eqref{delta},  we get
$$| \widetilde{J}_k^{\varepsilon}(x)|\leq C_T\frac{ \exp\left( \frac{L}{2}\left(1+\frac{\sqrt 3 }{\sqrt{2}}\right)\frac{|M|^{1/3}}{\eps^{1/3}}+aB(C_1-C_2)-\pi a\lambda_k +\frac{\widehat{L}}{\eps^{1/4}}|\lambda_k|^{1/4}\right)}{\varepsilon^{-2/3} k^{-4}|x^2+\lambda_k^2|^{1/2}}.$$
\end{proof}

The last result allow us to prove Theorem \ref{Teo0}.
\begin{proof}[Proof of Theorem \ref{Teo0}] 
By Proposition \ref{ContObs}, it is enough to  prove the observability inequality  \eqref{inobs} 
for some constant $C = C(T, \varepsilon)$.
In order to do this, we take $\varphi_0 \in L^2(0,L)$. Without loss of generality we can assume that
$$\varphi_0(x)=\sum_{k=1}^Nc_ke_k(x), \quad c_k \in \RR, $$
so the corresponding solution of equation \eqref{EqAdj} is given by 
$$\varphi(t,x)=\sum_{k=1}^Nc_k\exp(-\lambda_k(T-t))e_k(x).$$

By (\ref{burda}) we have  $\widetilde{J}_k^{\varepsilon}\in L^2(\mathbb{R})$ and we know that $\widetilde{J}_k^{\varepsilon}(z)$ is an entire function of exponential type $T/2$; then the $L^2$-version of the Paley-Wiener theorem implies  that $\widetilde{J}_k^{\varepsilon}(z)$ is the analytic extension of the Fourier transform of some $\tilde{\eta}_k\in L^2(\mathbb{R)}$ with support in $[-T/2,T/2].$ So, we consider the function 
$$J_k^{\varepsilon}(z)=\frac{\exp(-iTz/2)}{\exp(-T\lambda_k/2)}\widetilde{J}_k^{\varepsilon}(z), \quad z\in \mathbb{C},$$
which is the analytic extension of the Fourier transform of
$$\eta_k(t):=\frac{\tilde{\eta}_k(t-T/2)}{\exp(-T\lambda_k/2)}, \quad x\in \mathbb{R}.$$
Finally we set $\psi_k(t):= \eta_k(T-t), $ therefore 
\begin{equation}  \label{delt}
\int_0^{T}\psi_k(t)\exp(-\lambda_j(T-t))dt=\delta_{j,k}, \quad \text{ for all } j,k \in \mathbb{N}.
\end{equation}

From \eqref{delt} we have
$$\frac{k\pi}{L}c_k=\int_0^T\varphi_{x}(t,0)\psi_k(t)dt,$$  
therefore
\begin{equation*}
\frac{k\pi}{L} |c_k|\leq \|\varphi_{x}(\cdot,0)\|_{L^2(0,T)}\|\psi_k\|_{L^2(0,T)}, 
\end{equation*}
and then
\begin{equation*}
\|\varphi(0,\cdot)\|_{L^2(0,L)}  \leq  \|\varphi_{x}(\cdot,0)\|_{L^2(0,T)}\sum_{k=1}^N \frac{L}{k\pi} \|\psi_k\|_{L^2(0,T)}\|e_k\|_{L^2(0,L)} \exp(-\lambda_k T).
\end{equation*}

Moreover, the Plancherel theorem, (\ref{lambdak}) and (\ref{burda}) imply that
\begin{eqnarray*}
\|\psi_k \|_{L^2(0,T)}&= & \|\eta_k \|_{L^2(0,T)}= \exp(T\lambda_k/2)\| \widetilde{J}_k^{\epsilon} \|_{L^2(\mathbb{R})} \\
&\leq & C_{T,\varepsilon} k^4  \exp\left(\frac{\widehat{L}}{\eps^{1/4}}|\lambda_k|^{1/4}+ \frac{ \tau \lambda_k}{2}  \right)  \left(\int_{\mathbb{R}}\frac{dx}{|x^2+\lambda_k^2|}\right)^{1/2}\\
&\leq & C_{T,\varepsilon} \lambda_k  |\lambda_k|^{-1/2}   \exp\left(\frac{\widehat{L}}{\eps^{1/4}}|\lambda_k|^{1/4}
+\frac{\tau \lambda_k}{2}
\right), \quad k\geq 1.
\end{eqnarray*}

When $M>0$, we use (\ref{lambdak}) and (\ref{normabase}) to obtain
\begin{eqnarray*}
\|\varphi(0,\cdot)\|_{L^2(0,L)}  & \leq & C_{T,\varepsilon} \|\varphi_{x}(\cdot,0)\|_{L^2(0,T)}\sum_{k=1}^N \frac{|\lambda_k|^{1/2}}{k} \exp\left(\frac{\widehat{L}}{\eps^{1/4}}|\lambda_k|^{1/4}-(T-\tau/2) \lambda_k\right)\\
& \leq & C_{T,\varepsilon} \|\varphi_{x}(\cdot,0)\|_{L^2(0,T)}\sum_{k=1}^N  k\exp\left(-\frac{T-\tau/2}{2} \lambda_k\right)\leq  \|\varphi_{x}(\cdot,0)\|_{L^2(0,T)}\sum_{k=1}^{\infty}\frac{C_{T,\varepsilon}}{k^3}.
\end{eqnarray*}
\begin{equation*}
\end{equation*}

The case $M<0$ is proved similarly. 
\end{proof}

Once we know that system is controllable, in the next result we will establish more accurate estimates in order to prove the stated results 
about the cost of controllability. 

\begin{proposition}\label{precise}
There exists a constant $\tau >0$ such that \\
1.- For $M>0,$ $k\in \mathbb{N}$, we have
$$\frac{L}{2} \left(1+\frac{\sqrt 3 }{\sqrt{2}}\right) \frac{M^{1/3}}{\eps^{1/3}} +\frac{2^{-2/3}}{\eps^{1/3}}\frac{L^{4/3}}{(T-\tau)^{1/3}}\frac{\pi^{-1}(C_1-C_2)}{ (1+\sqrt{2})^2} -(T-\frac{\tau}{2})\lambda_k+\frac{\eps^{-1/4}L}{(1+\sqrt{2})}|\lambda_k|^{1/4}\leq-\frac{\tau}{2}\lambda_k, $$
 provided that
$$T>c_+L/M\quad \text{ with }  c_+ \sim 4,57.$$ 
2.- For $M<0,$ $k\in \mathbb{N}$, we have 
$$\frac{L}{2} \left(2+\frac{\sqrt 3 }{\sqrt{2}}\right) \frac{|M|^{1/3}}{\eps^{1/3}} +\frac{2^{-2/3}}{\eps^{1/3}}\frac{L^{4/3}}{(T-\tau)^{1/3}}\frac{\pi^{-1}(C_1-C_2)}{ (1+\sqrt{2})^2} -(T-\frac{\tau}{2})\lambda_k+\frac{\eps^{-1/4}L}{(1+\sqrt{2})}|\lambda_k|^{1/4}\leq-\frac{\tau}{2}\lambda_k, $$
 provided that
$$T>c_-L/|M|\quad \text{ with }  c_- \sim 6,19.$$ 
\end{proposition}
\begin{proof}
1.- We choose $\tau >0$ small enough in such a way $T-\tau>c_+L/M$. The function $r(x)= -(T-\tau)x+ \frac{1}{1+\sqrt{2}}L\eps^{-1/4}x^{1/4}$ is  decreasing for $x^{3/4}\geq  L/(4(1+\sqrt{2})(T-\tau)\eps^{1/4})$. Since 
$$\lambda_k\geq \frac{5}{16}M^{4/3}\eps^{-1/3},$$
we have $r(\lambda_k)\leq r\left(\frac{5}{16}M^{4/3}\eps^{-1/3}\right)$. So it is enough to show that 
$$\frac{L}{2} \left(1+\frac{\sqrt 3 }{\sqrt{2}}\right) \frac{|M|^{1/3}}{\eps^{1/3}} +\frac{1}{\eps^{1/3}}\frac{L^{4/3}}{(T-\tau)^{1/3}}\frac{C_1-C_2}{ 2^{2/3}(1+\sqrt{2})^2\pi} +r\left(\frac{5}{16}M^{4/3}\eps^{-1/3}\right)\leq 0,$$
which is equivalent to prove that
$$\left(1+\frac{\sqrt 3 }{\sqrt{2}}+\frac{1}{1+\sqrt{2}}5^{1/4}\right)X+\frac{2^{1/3}(C_1-C_2)}{ (1+\sqrt{2})^2\pi}-\frac{5}{8}X^4\leq 0,$$
where $X^{3}=M(T-\tau)/L$. For $M(T-\tau)>Lc_{+}$ the last inequality holds.

2.- We choose $\tau >0$ small enough in such a way $T-\tau>c_-L/M$ and proceeding as before we must have
$$\left(2+\frac{\sqrt 3 }{\sqrt{2}}+\frac{1}{1+\sqrt{2}}5^{1/4}\right)X+\frac{2^{1/3}(C_1-C_2)}{ (1+\sqrt{2})^2\pi}-\frac{5}{8}X^4\leq 0,$$
where $X^{3}=|M|(T-\tau)/L$. For $|M|(T-\tau)>Lc_{-}$ the last inequality holds.
%
%
\end{proof}

Using the previous results we can prove Theorem \ref{MainT}. 
%
\begin{proof}[Proof of Theorem \ref{MainT}]

Case 1. Let $M>0$. We proceed as in the proof of Theorem \ref{Teo0} to get
$$\|\varphi(0,\cdot)\|_{L^2(0,L)}  \leq C \|\varphi_{x}(\cdot,0)\|_{L^2(0,T)}\ \sum_{k=1}^N \frac{ \exp\left( \frac{L}{2}\left(1+\frac{\sqrt 3 }{\sqrt{2}}\right)\frac{|M|^{1/3}}{\eps^{1/3}}+aB(C_1-C_2)-\lambda_k (T-\frac{\tau}{2} )+\frac{\widehat{L}}{\eps^{1/4}}|\lambda_k|^{1/4}\right)}{|\lambda_k|^{1/2}\varepsilon^{-2/3}k^{-3}}.$$ 

From the estimates
$$\widehat{L}-\frac{L}{1+\sqrt{2}}\leq 2\alpha \eps^{1/4}\leq 2\alpha,\quad \max( \pi^4k^4L^{-4}\eps,\frac{5}{16}M^{4/3}\eps^{-1/3})\leq \lambda_k , \,\, k\geq 1,           $$
\begin{eqnarray*}
aB-\frac{1}{\eps^{1/3}}\frac{L^{4/3}}{(T-\tau)^{1/3}}\frac{2^{-2/3}\pi^{-1}}{ (1+\sqrt{2})^2}&=& \frac{2^{-2/3}(T-\tau)^{-1/3}\eps^{-1/3}[(L+\alpha \eps^{1/4})^{4/3}-L^{4/3}]}{(1+\sqrt{2})^{2}}   \\
& \leq & C_T \alpha \eps^{-1/12}(L+\alpha)^{1/3}\leq C_{T,L,M} |\lambda_k|^{1/4},
\end{eqnarray*}
and the first part of Proposition \ref{precise} it follows that
\begin{eqnarray*}
\|\varphi(0,\cdot)\|_{L^2(0,L)}  &\leq & C\varepsilon^{2/3} \|\varphi_{x}(\cdot,0)\|_{L^2(0,T)} \sum_{k=1}^N k^3\frac{\exp(C|\lambda_k|^{1/4} -\tau \lambda_k/2)}{|\lambda_k|^{1/2}} \\ 
&\leq & C\varepsilon^{1/6}\exp(-\widehat{C}/\eps^{1/3})\|\varphi_{x}(\cdot,0)\|_{L^2(0,T)} \sum_{k=1}^{\infty}k     \exp( -\tau \lambda_k/6)\\  
&\leq & C\varepsilon^{-5/6}\exp(-\widehat{C}/\eps^{1/3})\|\varphi_{x}(\cdot,0)\|_{L^2(0,T)} \sum_{k=1}^{\infty}\frac{1}{k^3} 
\end{eqnarray*}
where $\widehat{C}$ is a positive constant.\\

Case 2. If $M<0$ then $\|e_k\|_{L^2(0,L)}\leq C \exp\left(2^{-1}|M|^{1/3}L\eps^{-1/3} \right)$ for all $k\geq 1$. Thus, from the second part of Proposition \ref{precise} we obtain
$$\|\varphi(0,\cdot)\|_{L^2(0,L)} \leq C\varepsilon^{-5/6}\exp(-\widetilde{C}/\eps^{1/3})\|\varphi_{x}(\cdot,0)\|_{L^2(0,T)} \sum_{k=1}^{\infty}\frac{1}{k^3}$$
where $\widetilde{C}$ is a positive constant.

\end{proof}

\section{Lower bounds for the null optimal control} \label{lowbounds}
In this section we prove the existence of lower bounds for the null control.
\begin{proof}[Proof of Theorem \ref{MainT2}]
We set
$$y_0(x)=\exp\left(\frac{M^{1/3}}{2\eps^{1/3}}x\right)\sin\left(\frac{\pi x}{L}\right),$$
and we consider the null optimal control $u(t)\in L^2(0,T)$ for $y_0$. Therefore,
\begin{equation}\label{continuidad}
\|u\|_{L^2(0,T)}\leq K \|y_0\|_{L^2(0,L)}
\end{equation}
where $K=K(\eps, M,L,T)$ is the null optimal control constant.\\

An easy computation shows that
$$\|y_0\|_{L^2(0,L)}^2=2M^{-1/3}\eps^{1/3}\frac{|e^{\pi a}-1|}{a^2+4}$$
where $a=\pi^{-1}M^{1/3}L\eps^{-1/3}.$\\
%
%
From Definition \ref{defcont}, we have that 
\begin{equation*}  
\int_0^L y_0(x) \varphi(0,x) dx  =   - \int_0^T u(t) \varphi_{x}(t,0) dt.
\end{equation*}
 for each solution of the  adjoint  system \eqref{EqAdj}.
 
In particular, we can consider the solution of \eqref{EqAdj} given by 
$$\varphi(t,x)=e^{\lambda_k t}e_k(x), \quad \text{ for }(t,x)\in [0,T]\times[0,L],$$
from which we obtain
\begin{equation}\label{ortogonalidad}
\int_0^L \sin\left( \frac{k\pi x}{L}\right)  \sin\left(\frac{\pi x}{L}\right)dx= -  \frac{k\pi }{L}\int_0^T u(t) e^{\lambda_k t} dt 
\end{equation}
for all $k\geq 1$.\\
Now we introduce the entire function
\begin{equation}\label{uve}
v(s):= \int_{-T/2}^{T/2}u( t+T /2)e^{-ist}dt.
\end{equation}

From (\ref{ortogonalidad}) it follows that
$$v(i\lambda_k)=0, \, k\geq 2, \text{   and  }   v(i\lambda_1)= \frac{L^2}{2\pi } e^{-\lambda_1T/2}.$$ 
Moreover, the Holder inequality and (\ref{continuidad}) imply
\begin{eqnarray} \label{uvest}
|v(s)|&\leq & \exp(|\Im(s)|T/2)\int_0^T |u(t)|dt \notag \\
&\leq & KT^{1/2}  \|y_0\|_{L^2(0,L)} \exp(|\Im(s)|T/2)   \notag    \\ 
&=&\sqrt{2}KT^{1/2}M^{-1/6}\eps^{1/6}\left(\frac{|e^{\pi a}-1|}{a^2+4}\right)^{1/2} \exp(|\Im(s)|T/2). \notag \\
&\leq & \left\{ \begin{array}{ll}
\sqrt{2}KT^{1/2}M^{-1/6}\eps^{1/6}\frac{e^{\pi a/2}}{(a^2+4)^{1/2}} \exp(|\Im(s)|T/2),&M>0,\\
\sqrt{2}KT^{1/2}|M|^{-1/6}\eps^{1/6} (a^2+4)^{-1/2}\exp(|\Im(s)|T/2),&M<0.  \\
\end{array} \right .
\end{eqnarray} 
\textbf{Case I} $ M>0$.
Consider the entire function
$$f(s):=v\left(\frac{5}{16} \frac{s-iM^{4/3}}{\eps^{1/3}}  \right),\quad s\in \mathbb{C}.$$
Thus, 
\begin{equation}\label{buno}
f(b_k)=0, \, k\geq 2, \quad \text{and } f(b_1)=\frac{L^2}{2\pi }e^{-\lambda_1T/2}. 
\end{equation}
where
$$b_k=i\left(\frac{16}{5}\eps^{1/3}\lambda_k+M^{4/3}\right), \, k\geq 1.$$

From (\ref{uvest}) we have
\begin{equation}\label{belowend}
|f(s)|\leq \sqrt{2}KT^{1/2} M^{-1/6}\eps^{1/6}\frac{e^{\pi a/2}}{(a^2+4)^{1/2}}\exp\left(\frac{5 T}{32\eps^{1/3}}|\Im(s)-M^{4/3}|\right),
\end{equation}
so $f(z)$ is an entire function of exponential type $5 T\eps^{-1/3}/32$ on $\mathbb{C}^+,$ therefore we have the following representation (see \cite[page 56]{koosis1})
\begin{equation}\label{representacion}
\ln |f(s)|\leq \sum_{\ell =1}^{\infty}\ln \left|\frac{s-a_{\ell}}{s-\overline{a_{\ell}}}\right|+\sigma \Im(s)+\frac{\Im(s)}{\pi} \int_{-\infty}^{\infty}\frac{\ln |f(\tau)|}{|\tau -s|^2} d\tau, \quad s\in \mathbb{H},
\end{equation}
where $(a_{\ell})_{\ell}$ is the sequence of zeros of $f$ in $\mathbb{C}^+,$ each zero repeated many times as its multiplicity, and $\sigma$ is a real number satisfying
$$\sigma \leq \frac{5T}{32}\eps^{-1/3} .$$
By using (\ref{belowend}) we have
\begin{equation}\label{jota}
\frac{\Im(b_1)}{\pi} \int_{-\infty}^{\infty}\frac{\ln |f(\tau)|}{|\tau -b_1|^2} d\tau \leq \ln\left(\frac{\sqrt{2}KT^{1/2} \eps^{1/6}}{M^{1/6}(a^2+4)^{1/2}}\right)+\frac{a\pi}{2}+\frac{5 T}{32\eps^{1/3}}M^{4/3}.
\end{equation}
By the other hand,
\begin{eqnarray*}
\sum_{\ell =1}^{\infty}\ln \left|\frac{b_1-a_{\ell}}{b_1-\overline{a_{\ell}}}\right| &\leq & \sum_{k =2}^{\infty}\ln \left|\frac{b_1-b_k}{b_1-\overline{b_k}}\right| \\
&=&\sum_{k =2}^{\infty}\ln\left( \frac{\eps^{4/3} \pi^4(k^4-1)+3M^{2/3}\eps^{2/3}\pi^2(k^2-1)L^2/2}{\eps^{4/3} \pi^4(k^4+1)+3M^{2/3}\eps^{2/3}\pi^2(k^2+1)L^2/2+5M^{4/3}L^4/4}\right)\\
&\leq & \int_{2}^{\infty}\ln\left(\frac{\eps^{4/3} \pi^4 x^4+3M^{2/3}\eps^{2/3}\pi^2 x^2L^2/2}{\eps^{4/3} \pi^4 x^4+3M^{2/3}\eps^{2/3}\pi^2 x^2L^2/2+5M^{4/3}L^4/4}\right) dx \\
&=& \frac{L}{\eps^{1/3}\pi}  \int_{2\eps^{1/3}\pi L^{-1}}^{\infty}\ln\left(\frac{ x^4+3M^{2/3} x^2/2}{ x^4+3M^{2/3} x^2/2+5M^{4/3}/4}\right) dx:=J.
\end{eqnarray*}
We set 
$$\delta=2\eps^{1/3}\pi L^{-1}, \quad \gamma=\frac{3}{2}M^{2/3},\quad h(x)=\frac{  2x^2+\gamma}{(x^4+\gamma x^2+5\gamma ^{2}/9)(x^2+\gamma)}.$$ 
Integrating by parts and using the residue theorem we get
\begin{eqnarray}\label{calculajota}
J&=& \frac{L}{\eps^{1/3}\pi}  \left[ \delta \ln\left(1+ \frac{5\gamma^2/9}{\delta^4+\gamma\delta^2}  \right) -\frac{10\gamma^2}{9}\int_{\delta}^{\infty}h(x) dx\right]       \\
&=& 2 \ln\left(1+ \frac{5\gamma^2/9}{\delta^4+\gamma\delta^2}  \right) +\frac{10\gamma^2L}{9\eps^{1/3}\pi}\left(-\int_{0}^{\infty}+\int_0^{\delta}\right) h(x)dx            \notag \\
&=& 2 \ln\left(1+ \frac{5\gamma^2/9}{\delta^4+\gamma\delta^2}  \right) -\frac{5\gamma^2L}{9\eps^{1/3}\pi}\int_{-\infty}^{\infty}h(x)dx+\frac{20\gamma^2}{9\delta}\int_0^{\delta} h(x)dx            \notag \\
&\leq & 2 \ln\left(1+ \frac{5\gamma^2/9}{\delta^4+\gamma\delta^2}  \right) -\frac{\gamma^{1/2}L}{3\eps^{1/3} }\left(  \sqrt{3(3+2\sqrt{5})}-3   \right) + C  \notag 
\end{eqnarray}
where $C$ is a constant that does not depend on $\delta$,  and we have used that 
$$\int_{-\infty}^{\infty}\widetilde{h}=\frac{1}{\gamma^{3/2}}\int_{-\infty}^{\infty}\frac{  2x^2+1}{(x^4+x^2+5/9)(x^2+1)}dx=\frac{3\pi}{5\gamma^{3/2}}\left(\sqrt{3(3+2\sqrt{5})}-3\right).$$\\


We set $s=b_1$ in (\ref{representacion}), and using (\ref{buno}), (\ref{jota}), (\ref{calculajota}) we have
$$\frac{L^2}{ \eps^{1/6}}M^{1/6}(a^2+4)^{1/2}\leq C\left(1+ \frac{5\gamma^2/9}{\delta^4+\gamma\delta^2}  \right)^2\frac{\exp\left(\frac{5M^{1/3}}{8\eps^{1/3}}\left( TM-\left(\frac{8\widehat{C}}{5\sqrt{6}}-\frac{4}{5}\right)L \right) \right)}{T^{-1/2}\exp \left( 
T\left(-\frac{\eps \pi^4}{L^4}-\frac{3M^{2/3}\eps^{1/3} \pi^2}{2L^2} \right)
\right)}
K$$
where $\widehat{C}=\sqrt{3(3+2\sqrt{5})}-3$.\\ 

\textbf{Case   II} $ M<0$.
Consider the entire function
$$\widetilde{f}(s):=v\left(\frac{5}{16} \frac{s}{\eps^{1/3}}  \right),\quad s\in \mathbb{C},$$
where $v$ is given in (\ref{uve}). Thus, 
\begin{equation}\label{Buno}
\widetilde{f}(b_k)=0, \, k\geq 2, \quad \text{and } \widetilde{f}(b_1)=\frac{L^2}{2\pi } e^{-\lambda_1T/2}. 
\end{equation}
where
$$b_k=i\frac{16}{5}\eps^{1/3}\lambda_k, \, k\geq 1.$$

From (\ref{uvest}) we have
\begin{equation}\label{Belowend}
|\widetilde{f}(s)|\leq \sqrt{2}KT^{1/2} \frac{\eps^{1/6}}{|M|^{1/6}(a^2+4)^{1/2}}\exp\left(\frac{5 T}{32\eps^{1/3}}|\Im(s)|\right),
\end{equation}
so $\widetilde{f}(z)$ is an entire function of exponential type $5 T\eps^{-1/3}/32$ on $\mathbb{C}^+,$ therefore we have the following representation.
\begin{equation}\label{Representacion}
\ln |\widetilde{f}(s)|\leq \sum_{\ell =1}^{\infty}\ln \left|\frac{s-a_{\ell}}{s-\overline{a_{\ell}}}\right|+\sigma \Im(s)+\frac{\Im(s)}{\pi} \int_{-\infty}^{\infty}\frac{\ln |\widetilde{f}(\tau)|}{|\tau -s|^2} d\tau, \quad s\in \mathbb{H},
\end{equation}
where $(a_{\ell})_{\ell}$ is the sequence of zeros of $\widetilde{f}$ in $\mathbb{H},$ each zero repeated many times as its multiplicity, and $\sigma$ is a real number satisfying
$$\sigma \leq \frac{5T}{32}\eps^{-1/3} .$$
By using (\ref{Belowend}) we have
\begin{equation}\label{Jota}
\frac{\Im(b_1)}{\pi} \int_{-\infty}^{\infty}\frac{\ln |\widetilde{f}(\tau)|}{|\tau -b_1|^2} d\tau \leq \ln\left(\frac{\sqrt{2}KT^{1/2} \eps^{1/6}}{|M|^{1/6}(a^2+4)^{1/2}}\right). 
\end{equation}  
By the other hand,
\begin{eqnarray*}
\sum_{\ell =1}^{\infty}\ln \left|\frac{b_1-a_{\ell}}{b_1-\overline{a_{\ell}}}\right| &\leq & \sum_{k =2}^{\infty}\ln \left|\frac{b_1-b_k}{b_1-\overline{b_k}}\right| \\
&=&\sum_{k =2}^{\infty}\ln\left( \frac{\eps^{4/3} \pi^4(k^4-1)+3M^{2/3}\eps^{2/3}\pi^2(k^2-1)L^2/2}{\eps^{4/3} \pi^4(k^4+1)+3M^{2/3}\eps^{2/3}\pi^2(k^2+1)L^2/2+5M^{4/3}L^4/8}\right)\\
&\leq & \int_{2}^{\infty}\ln\left(\frac{\eps^{4/3} \pi^4 x^4+3M^{2/3}\eps^{2/3}\pi^2 x^2L^2/2}{\eps^{4/3} \pi^4 x^4+3M^{2/3}\eps^{2/3}\pi^2 x^2L^2/2+5M^{4/3}L^4/8}\right) dx \\
&=& \frac{L}{\eps^{1/3}\pi}  \int_{2\eps^{1/3}\pi L^{-1}}^{\infty}\ln\left(\frac{ x^4+3M^{2/3} x^2/2}{ x^4+3M^{2/3} x^2/2+5M^{4/3}/8}\right) dx:=\widetilde{J}.
\end{eqnarray*}
We set 
$$\delta=2\eps^{1/3}\pi L^{-1}, \quad \gamma=\frac{3}{2}M^{2/3},\quad \widetilde{h}(x)=\frac{  2x^2+\gamma}{(x^4+\gamma x^2+5\gamma^{2}/18)(x^2+\gamma)}.$$ 
Integrating by parts and using the residue theorem we get
\begin{eqnarray}\label{Calculajota}
\widetilde{J}&=& \frac{L}{\eps^{1/3}\pi}  \left[ \delta \ln\left(1+ \frac{5\gamma^2/18}{\delta^4+\gamma\delta^2}  \right) -\frac{5\gamma^2}{9}\int_{\delta}^{\infty}\widetilde{h}(x) dx\right]       \\
&=& 2 \ln\left(1+ \frac{5\gamma^2/18}{\delta^4+\gamma\delta^2}  \right) +\frac{5\gamma^2L}{9\eps^{1/3}\pi}\left(-\int_{0}^{\infty}+\int_0^{\delta}\right) \widetilde{h}(x)dx            \notag \\
&=& 2 \ln\left(1+ \frac{5\gamma^2/18}{\delta^4+\gamma\delta^2}  \right) -\frac{5\gamma^2L}{18\eps^{1/3}\pi}\int_{-\infty}^{\infty}\widetilde{h}(x)dx+\frac{10\gamma^2}{9\delta}\int_0^{\delta} \widetilde{h}(x)dx            \notag \\
&\leq & 2 \ln\left(1+ \frac{5\gamma^2/18}{\delta^4+\gamma\delta^2}  \right) -\frac{\gamma^{1/2}L}{3\eps^{1/3} }\left(  \sqrt{3(3+\sqrt{10})}-3   \right)  + C  \notag 
\end{eqnarray}
where $C$ is a constant that does not depend on $\delta$, and we have used that 
$$\int_{-\infty}^{\infty}\widetilde{h}=\frac{1}{\gamma^{3/2}}\int_{-\infty}^{\infty}\frac{  2x^2+1}{(x^4+x^2+5/18)(x^2+1)}dx=\frac{6\pi}{5\gamma^{3/2}}\left(\sqrt{3(3+\sqrt{10})}-3\right).$$

We set $s=b_1$ in (\ref{representacion}), and using (\ref{Buno}), (\ref{Jota}), (\ref{Calculajota}) we have
$$\frac{L^2}{ \eps^{1/6}}|M|^{1/6}(a^2+4)^{1/2}\leq C\left(1+ \frac{5\gamma^2/18}{\delta^4+\gamma\delta^2}  \right)^2\frac{\exp\left(\frac{5|M|^{1/3}}{16\eps^{1/3}}\left( T|M|-\left(\frac{16\widetilde{C}}{5\sqrt{6}}\right)L \right) \right)}{T^{-1/2} 
\exp \left(  T\left(-\frac{\eps \pi^4}{L^4}-\frac{3M^{2/3}\eps^{1/3} \pi^2}{2L^2} \right)
\right)}
K$$
where $\widetilde{C}=\sqrt{3(3+\sqrt{10})}-3$. 
\end{proof}

%
\bibliographystyle{amsplain}
\bibliography{biblio_uniform}

\end{document}